\documentclass[11pt]{article}
\usepackage[T1]{fontenc}
\usepackage{hyperref, amsfonts, amsmath, amssymb, amsthm, fullpage, mathtools, microtype, enumitem, tikz, mleftright}
\usepackage[noabbrev, capitalize]{cleveref}

\title{Bounds on median eigenvalues of graphs of bounded degree}
\author{
    Hricha Acharya\thanks{School of Mathematical and Statistical Sciences, Arizona State University, Tempe, AZ 85281, USA. Email: {\tt \{hachary3, zilinj\}@asu.edu}.}
    \and Zilin Jiang\footnotemark[1]${\ }^{,}$\thanks{School of Computing and Augmented Intelligence, Arizona State University, Tempe, AZ 85281, USA. Supported in part by the Simons Foundation through its Travel Support for Mathematicians program and by U.S.\ taxpayers through NSF grant 2451581.}
    \and Shengtong Zhang\thanks{Department of Mathematics, Stanford University, Stanford, CA 94305, USA. Email: {\tt stzh1555@stanford.edu}}
}
\date{}

\newtheorem{theorem}{Theorem}[section]
\newtheorem{corollary}[theorem]{Corollary}
\newtheorem{lemma}[theorem]{Lemma}
\newtheorem{proposition}[theorem]{Proposition}

\newtheorem{conjecture}[theorem]{Conjecture}

\theoremstyle{definition}

\theoremstyle{remark}
\newtheorem*{remark}{Remark}

\DeclarePairedDelimiter{\abs}{\lvert}{\rvert}
\DeclareMathOperator{\tr}{tr}

\newcommand{\sset}[1]{\{#1\}}
\newcommand{\dset}[2]{\sset{#1 \colon #2}}
\newcommand{\lah}{\lambda_h}
\newcommand{\lal}{\lambda_\ell}

\newcommand{\eps}{\varepsilon}
\newcommand{\avge}{\varepsilon(G)}
\newcommand{\N}{\mathbb{N}}
\newcommand{\dbar}{\overline{d}}
\newcommand{\quadand}{\quad\text{and}\quad}

\newcommand{\EE}{\operatorname{E}}

\begin{document}

\maketitle

\begin{abstract}

We prove that for every integer $d \ge 3$, the median eigenvalues of any graph of maximum degree $d$ are bounded above by $\sqrt{d-1}$. We also prove that, in three separate cases, the median eigenvalues of a graph of maximum degree $d$ are bounded below by $-\sqrt{d-1}$: when the graph is triangle-free, when $d-1$ is a perfect square, or when $d \ge 75$. These results resolve, for all but finitely many values of $d$, an open problem of Mohar on median eigenvalues of graphs of maximum degree $d$. As a byproduct, we establish an upper bound on the average energy of graphs of maximum degree at most $d$, generalizing a previous result of van Dam, Haemers, and Koolen for $d$-regular graphs.

\end{abstract}

\section{Introduction} \label{sec:intro}

Spectral graph theory has traditionally focused on extremal portions of the spectrum, particularly the largest eigenvalue and, more generally, the leading eigenvalues. In contrast, eigenvalues near the center of the spectrum have received much less attention. Among these, the \emph{median eigenvalues} form a natural and important object of study, motivated both by intrinsic spectral questions and by applications arising from H\"uckel molecular orbital theory in chemistry \cite{LLSG13}.
Let $G$ be a simple graph of order $n$, and let $\lambda_1 \ge \lambda_2 \ge \cdots \ge \lambda_n$ denote the eigenvalues of its adjacency matrix $A_G$. The \emph{median eigenvalues} of $G$ are the eigenvalues $\lah$ and $\lal$, where 
\[
    h = \lfloor (n+1)/2\rfloor \quadand \ell = \lceil (n+1)/2\rceil.
\]

The systematic study of median eigenvalues was initiated by Fowler and Pisanski~\cite{FP10-b,FP10-a}, who conducted computational experiments on graphs of maximum degree at most three, also known as \emph{subcubic graphs}. They conjectured that, apart from finitely many exceptions, every subcubic graph has its median eigenvalues in the interval $[-1,1]$.

Subsequent work of Mohar~\cite{M13,M16} confirmed the conjecture of Fowler and Pisanski for \emph{bipartite} subcubic graphs, and identified the Heawood graph, namely the incidence graph of the Fano plane, as the unique exception, with median eigenvalues $\pm\sqrt{2}$. Recently Acharya, Jeter, and Jiang~\cite{AJJ25} completely resolved the conjecture, proving that every subcubic graph, except for the Heawood graph, has its median eigenvalues in $[-1,1]$.

With the subcubic case settled, we turn to graphs of maximum degree $d$ for a general integer $d$. Mohar~\cite{M15} proved that the median eigenvalues of any graph of maximum degree $d$ are at most $\sqrt d$ in absolute value, and conjectured that $\sqrt d$ could be improved to $\sqrt{d-1}$. Mohar also pointed out that the median eigenvalues of the incidence graph of a projective plane of order $d-1$ are equal to $\pm\sqrt{d-1}$; hence the bound $\sqrt{d-1}$ would be optimal whenever $d-1$ is a prime power.

In this paper, we confirm Mohar's conjecture for all but finitely many values of $d$. We begin with the upper bound on the median eigenvalues.

\begin{theorem} \label{thm:upper_bound}
    For every integer $d \ge 3$, the median eigenvalues of any graph of maximum degree $d$ are at most $\sqrt{d-1}$.
\end{theorem}

We then prove the lower bound on the median eigenvalues in three separate cases.

\begin{theorem} \label{thm:lower_bound_triangle_free}
    For every integer $d \ge 3$, the median eigenvalues of any triangle-free graph of maximum degree $d$ are at least $-\sqrt{d-1}$.
\end{theorem}

\begin{theorem}\label{thm:lower_bound_d-1_perfect_square}
    For every integer $d \ge 2$ such that $d-1$ is a perfect square, the median eigenvalues of any graph of average degree at most $d$ are at least $-\sqrt{d-1}$.
\end{theorem}

\begin{theorem}\label{thm:lower_bound_d_large}
    For every integer $d \ge 75$, the median eigenvalues of any graph of maximum degree $d$ are at least $-\sqrt{d-1}$.
\end{theorem}

\begin{remark}
    \Cref{thm:lower_bound_d-1_perfect_square} replaces the maximum-degree hypothesis with an average-degree hypothesis.
\end{remark}

A key ingredient in the proof of \cref{thm:lower_bound_d_large} is an upper bound on the \emph{average energy} of a graph $G$ on $n$ vertices, defined by
\[
    \avge = \frac{1}{n}\sum_{i=1}^n \abs{\lambda_i}.
\]
This notion is closely related to the \emph{energy} of a graph, namely $\sum_i \abs{\lambda_i}$, introduced by Gutman \cite{G78} in connection with H\"uckel molecular orbital theory. The first result on average energy dates back to McClelland \cite{Mc71}, who proved that the average energy of any graph of average degree $d$ is at most $\sqrt{d}$ (see \cref{thm:mc-clelland-bound}).

For every $d$-regular graph $G$, van Dam, Haemers, and Koolen \cite{DHK14} proved that its average energy is at most $\sqrt{d-1} + 1/(d + \sqrt{d-1})$, with equality if and only if $G$ is a vertex-disjoint union of incidence graphs of projective planes of order $d-1$, or, when $d = 2$, a vertex-disjoint union of triangles and hexagons\footnote{In \cite[Theorem 1.1]{DHK14}, the bound on the average energy for $d$-regular graphs is stated as $(d+(d^2-d)\sqrt{d-1})/(d^2-d+1)$, which simplifies to $\sqrt{d-1} + 1/(d + \sqrt{d-1})$.}. Here we generalize their result to graphs of bounded maximum degree, from which we obtain a bound on the median eigenvalues.

\begin{theorem}\label{thm:avg_energy}
    For every integer $d \ge 3$ and every graph $G$ of maximum degree $d$, its average energy satisfies
    \[
        \avge \le \sqrt{d-1} + \frac{1}{d + \sqrt{d-1}}.
    \]
\end{theorem}

\begin{corollary}\label{cor:median_energy_bound}
    For every integer $d\ge 3$, the median eigenvalues of any graph of maximum degree $d$ are at most $\sqrt{d-1}+1/(d+\sqrt{d-1})$ in absolute value.
\end{corollary}

The rest of the paper is organized as follows. In \cref{sec:bounds_median_ev} we prove \cref{thm:upper_bound,thm:lower_bound_triangle_free}. In \cref{sec:median_ev_avg_deg} we prove \cref{thm:lower_bound_d-1_perfect_square}, and in \cref{sec:bounds_avg_energy} we prove \cref{thm:avg_energy}. The proof of \cref{thm:lower_bound_d_large} is given in \cref{sec:lower_bound_d_large}. We conclude with some further remarks in \cref{sec:further_remarks}.

\section{Bounds on median eigenvalues} \label{sec:bounds_median_ev}

\begin{lemma} \label{lem:moments}
    For every $n$-vertex graph $G$ of average degree $\dbar$, its eigenvalues $\lambda_1, \dots, \lambda_n$ satisfy
    \[
        \sum_i \lambda_i = 0, \quad
        \sum_i \lambda_i^2 = \dbar n, \quad
        \sum_i \lambda_i^3 \ge 0, \quad
        \sum_i \lambda_i^4 \ge \mleft(2\dbar^2-\dbar\mright)n,
    \]
    and moreover $\sum_i \lambda_i^3 = 0$ if and only if $G$ is triangle-free.
\end{lemma}

\begin{proof}
    Notice that $\sum_i \lambda_i^k = \tr(A_G^k)$ counts the number of closed walks of length $k$ in $G$. Clearly, $\sum_i \lambda_i = 0$, and $\sum_i \lambda_i^2 = \dbar n$, since this quantity counts each edge of $G$ twice. The sum $\sum_i \lambda_i^3$ counts six times the number of triangles in $G$, and is therefore non-negative; moreover, it is equal to $0$ if and only if $G$ is triangle-free. Finally, by considering only the closed walks of the form $ababa$, $abaca$, or $abcba$, we obtain the lower bound via the Cauchy--Schwarz inequality:
    \begin{equation*}
        \sum_i \lambda_i^4 \ge \sum_i d_i^2 + d_i(d_i - 1) = \sum_i 2d_i^2 - d_i \ge \mleft(2\dbar^2-\dbar\mright)n,
    \end{equation*}
    where $(d_1, \dots, d_n)$ denotes the degree sequence of $G$.
\end{proof}

\begin{remark}
    Our bound on $\sum_i \lambda_i^4$ is inspired by van Dam, Haemers, and Koolen \cite[page 125]{DHK14}, where they prove the same bound for $d$-regular graphs.
\end{remark}

\begin{figure}
    \centering
    \begin{tikzpicture}[x=0.42cm, y=0.7cm]
        \draw[->] (-10.6,0) -- (11.7,0);
        \draw[dashed] (3,0) -- (3,5);
        \draw[dashed] (10,0) -- (10,5);
        \draw[dashed] (3,3.8025) -- (10,3.8025);
        \draw[thick] plot[smooth] coordinates {
            (-10,0.0000) (-9.5,0.4357) (-9,0.7245) (-8.5,0.8905) (-8,0.9563)
            (-7.5,0.9429) (-7,0.8700) (-6.5,0.7557) (-6,0.6165) (-5.5,0.4676)
            (-5,0.3225) (-4.5,0.1934) (-4,0.0907) (-3.5,0.0238) (-3,0.0000)
            (-2.5,0.0255) (-2,0.1050) (-1.5,0.2415) (-1,0.4365) (-0.5,0.6902)
            (0,1.0012) (0.5,1.3666) (1,1.7820) (1.5,2.2414) (2,2.7375)
            (2.5,3.2613) (3,3.8025) (3.5,4.3491) (4,4.8877) (4.5,5.4035)
            (5,5.8800) (5.5,6.2993) (6,6.6420) (6.5,6.8872) (7,7.0125)
            (7.5,6.9940) (8,6.8063) (8.5,6.4224) (9,5.8140) (9.5,4.9512)
            (10,3.8025) (10.5,2.3351) (11,0.5145) (11.125,0.0000)
        };
        \foreach \x/\y/\label in {-10/-0.12/$-d$, -3/-0.12/$-\eps_0$, 3/-0.2/$\eps_0$, 10/-0.12/$d$, 11.125/-0.26/$\alpha$} {
            \draw (\x,0.12) -- (\x,-0.12);
            \node[below] at (\x,\y) {\label};
        }
    \end{tikzpicture}
    \caption{The graph of the magic polynomial.} \label{fig:f-graph}
\end{figure}
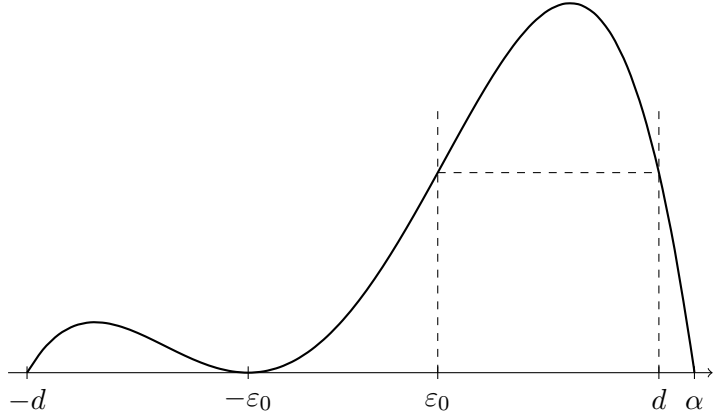

\begin{proof}[Proof of \Cref{thm:upper_bound}]
    Set $\eps_0 = \sqrt{d-1}$. Consider the magic polynomial $f$ defined by
    \[
        f(x) = (\alpha - x)(x+\eps_0)^2(x+d),
    \]
    where $\alpha \in \mathbb{R}$ is chosen so that $f(\eps_0) = f(d)$. Solving this equation yields
    \[
        \alpha = \frac{d^2 + 2d\eps_0 + 2\eps_0^2}{d + 2\eps_0}.
    \]
    Since $f(-d) = f(-\eps_0) = 0$, $f(\eps_0) = f(d)$, and $\alpha > d$, it follows that the graph of the quartic polynomial $f$ has the shape shown in \cref{fig:f-graph}. More precisely,
    \[
        f(x) \ge 0 \text{ for } x \in [-d, d] \quadand f(x) > f(d) \text{ for } x \in (\eps_0, d).
    \]
    
    We denote the coefficients of $f(x)$ by
    \[
        f(x) = -x^4 + c_3 x^3 + c_2 x^2 + c_1 x + c_0,
    \]
    and, by Vi\`ete's formulas,
    \[
        c_3 = \alpha - (d + 2\eps_0) < 0 \quadand c_2 = \alpha(d + 2\eps_0) - 2d\eps_0 - \eps_0^2 = d^2 + \eps_0^2 = d^2 + d - 1.
    \]

    Let $G$ be an $n$-vertex graph of maximum degree $d$, and let $\lambda_1, \dots, \lambda_n$ be the eigenvalues of $G$. We adopt a probabilistic point of view: let $\lambda$ be chosen uniformly at random from $\sset{\lambda_1, \dots, \lambda_n}$. Using \cref{lem:moments} and the fact that $c_3 < 0$, we estimate the expectation of $f(\lambda)$ as follows:
    \begin{equation} \label{eq:expectation_f_lambda}
        \EE f(\lambda) = -\frac{1}{n}\sum_i \lambda_i^4 + \frac{c_3}{n} \sum_i \lambda_i^3 + \frac{c_2}{n} \sum_i \lambda_i^2 + \frac{c_1}{n} \sum_i \lambda_i + c_0 \le -\dbar (2\dbar-1) + c_2 \dbar + c_0,
    \end{equation}
    which, as a quadratic function of $\dbar$, is increasing on $(-\infty, (1+c_2)/4]$. Since $c_2 = d^2 + d - 1 \ge 4d - 1$ for $d \ge 3$, we have $\dbar \le d \le (1+c_2)/4$, and therefore
    \[
        \EE f(\lambda) \le -d(2d-1) + c_2 d + c_0.
    \]

    We now consider the random variable $\mu$ with the following distribution:
    \[
        \Pr(\mu = -d) = \Pr(\mu = d) = \frac{1}{2(d^2 - d + 1)} \quadand
        \Pr(\mu = -\eps_0) = \Pr(\mu = \eps_0) = \frac{d^2 - d}{2(d^2 - d + 1)}.
    \]
    A direct computation shows that
    \begin{equation} \label{eq:moments_mu}
        \EE \mu = 0, \quad
        \EE \mu^2 = d, \quad
        \EE \mu^3 = 0, \quad
        \EE \mu^4 = d(2d-1).
    \end{equation}
    Using the facts that $f(-d) = f(-\eps_0) = 0$ and $f(\eps_0) = f(d)$, we have
    \[
        \EE f(\lambda) \le -d(2d-1) + c_2d + c_0 \stackrel{\eqref{eq:moments_mu}}{=} \EE f(\mu) = \Pr(\mu = \eps_0 \text{ or } \mu = d) f(d) = \frac{1}{2} f(d).
    \]
    On the other hand, since $\abs{\lambda_i} \le d$ for all $i$, $f(x) \ge 0$ for $x \in [-d, d]$, and $f(x) > f(d)$ for $x \in (\eps_0, d)$, we have
    \[
        \EE f(\lambda) \ge \Pr(\eps_0 < \lambda \le d) f(d).
    \]
    Therefore $\Pr(\eps_0 < \lambda \le d) \le 1/2$, and equality can hold only if $\Pr(\lambda = d) = 1/2$.
    
    We claim that $\Pr(\lambda = d) < 1/2$. Since $d$ is the largest eigenvalue of a connected component of $G$ if and only if that component is $d$-regular, the Perron--Frobenius theorem implies that the multiplicity of $d$ as an eigenvalue of $G$ is at most $n/(d+1)$, and hence $\Pr(\lambda = d) < 1/2$.
    
    Finally, we must have the strict inequality $\Pr(\eps_0 < \lambda \le d) < 1/2$, which implies that the median eigenvalues of $G$ are at most $\eps_0 = \sqrt{d-1}$.
\end{proof}

\begin{remark}
    The incidence graph $H$ of a projective plane of order $d-1$ has eigenvalues $\pm d$ with multiplicity $1$ each, and eigenvalues $\pm \sqrt{d-1}$ with multiplicity $d^2 - d$ each. The random variable $\mu$ can be interpreted as being chosen uniformly at random from the eigenvalues of $H$.
\end{remark}

One might be tempted to apply the same strategy to obtain a lower bound on the median eigenvalues by considering the expectation $\EE f(-\lambda)$ instead of $\EE f(\lambda)$, as this would bound the probability $\Pr(-d \le \lambda < -\eps_0)$. However, the term $(c_3/n)\sum_i \lambda_i^3$ in \eqref{eq:expectation_f_lambda} would become $-(c_3/n)\sum_i \lambda_i^3$, which is no longer guaranteed to be non-positive. Nevertheless, the triangle-free assumption circumvents this difficulty.

\begin{proof}[Proof of \Cref{thm:lower_bound_triangle_free}]
    The proof is nearly identical to that of \cref{thm:upper_bound}, except that we consider $\EE f(-\lambda)$ instead of $\EE f(\lambda)$ and modify the probabilistic argument accordingly. Since $G$ is triangle-free, \cref{lem:moments} gives $\sum_i \lambda_i^3 = 0$, so \eqref{eq:expectation_f_lambda} remains valid.
\end{proof}

\section{\texorpdfstring{Lower bound when $d-1$ is a perfect square}{Lower bound when d-1 is a perfect square}} \label{sec:median_ev_avg_deg}

To prove \cref{thm:lower_bound_d-1_perfect_square}, we use McClelland's upper bound on average energy \cite{Mc71}, together with a characterization of the extremal graphs.

\begin{theorem}[McClelland \cite{Mc71}]\label{thm:mc-clelland-bound}
    For every graph $G$ with average degree at most $d$, its average energy satisfies
    \[
        \avge \le \sqrt{d}
    \]
    with equality if and only if $G$ is an empty graph or a matching, and $d$ is equal to $0$ or $1$ respectively.
\end{theorem}

\begin{proof}
    Since $\sum_i \lambda_i^2 \le dn$ by \cref{lem:moments}, using the Cauchy--Schwarz inequality, we have
    \[
        \avge = \frac{1}{n}\sum_i \abs{\lambda_i} \le \sqrt{\frac{1}{n}\sum_i \lambda_i^2} \le \sqrt{d}.
    \]
    Equality holds if and only if $\abs{\lambda_i} = \sqrt{d}$ for all $i$, and $d$ is equal to the average degree of $G$.
    
    Suppose that all eigenvalues of $G$ have the same absolute value. Let $C$ be an arbitrary connected component of $G$. By the Perron--Frobenius theorem, the largest eigenvalue of $C$ has multiplicity $1$. If all eigenvalues of $C$ have the same absolute value and still sum to $0$, then $C$ must have order at most $2$; that is, $C$ is either a single vertex or a single edge. Since this holds for every connected component of $G$, the graph $G$ is either an empty graph or a matching.
\end{proof}

The main idea in the proofs of \cref{thm:lower_bound_d-1_perfect_square} is to combine average-energy estimates with an integrality argument. We now present the proof of \cref{thm:lower_bound_d-1_perfect_square}.

\begin{proof}[Proof of \Cref{thm:lower_bound_d-1_perfect_square}]
    Fix an integer $d \ge 2$ such that $d-1$ is a perfect square. Let $G$ be a graph of average degree at most $d$, and let $\eps = \avge$ be the average energy of $G$. Set
    \[
        \eps_0 = \sqrt{d-1} \in \N^+ \quadand \eps_1 = \sqrt{d}.
    \]
    We shall repeatedly use the identity
    \begin{equation} \label{eq:eps_1_eps_0_identity}
        (\eps_1 - \eps_0)(\eps_1 + \eps_0) = \eps_1^2 - \eps_0^2 = 1.
    \end{equation}
    Assume for the sake of contradiction that the lower median eigenvalue $\lal$ of $G$ is less than $-\eps_0$.
    
    Let $I$ denote the set of indices $i$ such that $\lambda_i \neq -\eps_0$. Since $\prod_{i\in I}(x - \lambda_i) \in \mathbb{Z}[x]$, we deduce that $\prod_{i \in I}(\eps_0 + \lambda_i)$ is a nonzero integer, and in particular,
    \begin{equation} \label{eq:sq_B_estimate}
        B := \prod_{i\in I} \abs*{\eps_0 + \lambda_i} \ge 1.
    \end{equation}
    We shall derive a contradiction by showing that the last inequality fails. We partition $I$ into two subsets $J$ and $K$ as follows:
    \[
        J = \dset{i}{\lambda_i < -\eps_0} \quadand K = \dset{i}{\lambda_i > -\eps_0}.
    \]
    Notice that $J \supseteq \sset{\ell, \dots, n}$, and in particular,
    \begin{equation} \label{eqn:sq_J_estimate}
        \abs{J} \ge n/2.
    \end{equation}
    This partition of $I = J \cup K$ allows us to factor $B$ as
    \[
        B_J := \prod_{i\in J} \abs*{\eps_0 + \lambda_i}
        \quadand 
        B_K := \prod_{i\in K} \abs*{\eps_0 + \lambda_i}.
    \]

    We first estimate $B_J$. By the inequality of arithmetic and geometric means, we have
    \begin{equation} \label{eq:sq_B_J_estimate}
        B_J \le \mleft(\frac{1}{\abs{J}} \sum_{i \in J} \abs*{\eps_0 + \lambda_i}\mright)^{\abs{J}}.
    \end{equation}
    Since $\lambda_i < -\eps_0 < 0$ for all $i \in J$, and $\sum_i \lambda_i = 0$ by \cref{lem:moments}, we have
    \begin{multline*}
        \sum_{i \in J} \abs*{\eps_0 + \lambda_i} = -\sum_{i \in J} (\lambda_i + \eps_0) = \abs*{\sum_{i \in J} \lambda_i} - \eps_0 \abs{J}
        = \frac{1}{2}\mleft(\abs*{\sum_{i \in J} \lambda_i} + \abs*{\sum_{i \notin J} \lambda_i}\mright) - \eps_0 \abs{J} \\
        \le \frac{1}{2}\sum_i \abs{\lambda_i} - \eps_0 \abs{J}
        = \frac{\eps n}{2} - \eps_0 \abs{J} \stackrel{\eqref{eqn:sq_J_estimate}}{\le} (\eps - \eps_0) \abs{J}.
    \end{multline*}
    Substituting this into \eqref{eq:sq_B_J_estimate}, we obtain
    \begin{equation} \label{eq:sq_B_J_estimate_2}
        B_J \le \mleft(\eps - \eps_0\mright)^{\abs{J}} < (\eps_1 - \eps_0)^{\abs{J}},
    \end{equation}
    where the last inequality follows from \cref{thm:mc-clelland-bound}.

    Now we turn to $B_K$. By the inequality of quadratic and geometric means, we have
    \begin{equation} \label{eq:sq_B_K_estimate}
        B_K \le \mleft(\frac{1}{\abs{K}} \sum_{i \in K} (\eps_0 + \lambda_i)^2\mright)^{\abs{K}/2}.
    \end{equation}
    We estimate two different linear combinations of \[
        S_J := \sum_{i \in J} (\eps_0 + \lambda_i)^2
        \quadand 
        S_K := \sum_{i \in K} (\eps_0 + \lambda_i)^2.
    \]
    For the first linear combination, since $\sum_i \lambda_i = 0$ and $\sum_i \lambda_i^2 = dn = \eps_1^2n$ by \cref{lem:moments}, we have
    \begin{equation} \label{eq:sq_B_K_estimate_1}
        S_J + S_K \le \sum_i (\eps_0 + \lambda_i)^2 = \eps_0^2 n + 2\eps_0 \sum_i \lambda_i + \sum_i \lambda_i^2 = (\eps_0^2 + \eps_1^2) n \stackrel{\eqref{eqn:sq_J_estimate}}{\le} (\eps_0^2 + \eps_1^2) \cdot 2\abs{J}.
    \end{equation}
    For the second linear combination, applying the inequality of arithmetic and geometric means, we have
    \begin{multline*}
        (\eps_1 + \eps_0)^4 S_J + S_K
        \ge \abs{I} \mleft(\prod_{i \in J} (\eps_1 + \eps_0)^4(\eps_0 + \lambda_i)^2 \cdot \prod_{i \in K} (\eps_0 + \lambda_i)^2 \mright)^{1/\abs{I}} \\
        \stackrel{\eqref{eq:sq_B_estimate}}{\ge} \abs{I} \mleft((\eps_1 + \eps_0)^{4\abs{J}}\mright)^{1/\abs{I}}
        = \frac{\abs{I}}{2\abs{J}} \mleft((\eps_1 + \eps_0)^2\mright)^{\frac{2\abs{J}}{\abs{I}}} \cdot 2\abs{J}.
    \end{multline*}
    One checks that the function $x \mapsto a^x/x$ is increasing on $[1/\ln a, \infty)$. As $(\eps_1 + \eps_0)^2 \ge \mleft(\sqrt{2} + 1\mright)^2 > e$ and $2\abs{J} / \abs{I} \ge 1$ by \eqref{eqn:sq_J_estimate}, we have
    \begin{equation} \label{eq:sq_B_K_estimate_2}
        (\eps_1 + \eps_0)^4 S_J + S_K \ge (\eps_1 + \eps_0)^2 \cdot 2\abs{J}.
    \end{equation}
    Combining \eqref{eq:sq_B_K_estimate_1} and \eqref{eq:sq_B_K_estimate_2}, we can estimate $S_K$ as follows:
    \begin{multline*}
        S_K \le \frac{\mleft(\eps_1+\eps_0\mright)^4\mleft(\eps_1^2 + \eps_0^2\mright) - (\eps_1 + \eps_0)^2}{\mleft(\eps_1+\eps_0\mright)^4 - 1} \cdot 2\abs{J} \\
        \stackrel{\eqref{eq:eps_1_eps_0_identity}}{=} \frac{(\eps_1 + \eps_0)^2(\eps_1^2 + \eps_0^2) - (\eps_1 + \eps_0)^2(\eps_1 - \eps_0)^2}{(\eps_1 + \eps_0)^2 - (\eps_1 - \eps_0)^2} \cdot 2\abs{J} = (\eps_1 + \eps_0)^2 \abs{J}.
    \end{multline*}
    Substituting this into \eqref{eq:sq_B_K_estimate}, we obtain
    \[
        B_K \le \mleft(\frac{(\eps_1 + \eps_0)^2 \abs{J}}{\abs{K}}\mright)^{\abs{K}/2} = \mleft(\mleft(\frac{(\eps_1 + \eps_0)^2}{\abs{K}/\abs{J}}\mright)^{\abs{K}/\abs{J}}\mright)^{\abs{J}/2}
    \]
    One checks that the function $x \mapsto (a/x)^x$ is increasing on $(0, a/e)$. As $(\eps_1 + \eps_0)^2 > e$ and $\abs{K}/\abs{J} \le 1$ by \eqref{eqn:sq_J_estimate}, we have
    \begin{equation} \label{eq:sq_B_K_estimate_3}
        B_K \le \mleft(\mleft(\eps_1 + \eps_0\mright)^2\mright)^{\abs{J}/2} = (\eps_1 + \eps_0)^{\abs{J}}.
    \end{equation}

    Combining \eqref{eq:sq_B_J_estimate_2} and \eqref{eq:sq_B_K_estimate_3}, we obtain
    \[
        B_J B_K < (\eps_1 - \eps_0)^{\abs{J}} (\eps_1 + \eps_0)^{\abs{J}} \stackrel{\eqref{eq:eps_1_eps_0_identity}}{=} 1,
    \]
    which contradicts \eqref{eq:sq_B_estimate}.
\end{proof}

\section{Upper bound on average energy} \label{sec:bounds_avg_energy}

To prove \cref{thm:avg_energy}, the key analytic ingredient is a carefully chosen polynomial with no $x^3$ term.

\begin{lemma}\label{lem:polynomial_bound}
    Let $d$ be an integer at least $2$, and define
    \[
        f(x) = \mleft(x - d\mright)\mleft(x - \sqrt{d - 1}\mright)^2 \mleft(x + d + 2\sqrt{d - 1}\mright).
    \]
    Then, for every graph $G$ of maximum degree $d$, the spectrum $\sset{\lambda_i}$ and the average energy $\avge$ of $G$ satisfy
    \begin{equation} \label{eq:polynomial_bound}
        \frac{1}{n}\sum_i f(\abs{\lambda_i}) \ge 2\sqrt{d-1}\mleft(d + \sqrt{d-1}\mright)^2 \mleft(\avge - \sqrt{d-1} - \frac{1}{d + \sqrt{d-1}} \mright).
    \end{equation}
\end{lemma}
\begin{proof}
    Expanding $f(x)$ yields $x^4 - \alpha x^2 + \beta x - \gamma$, where the coefficients are defined by
    \[
        \alpha = d^2 + 3d - 3 + 2d\sqrt{d-1}, \quad \beta = 2\sqrt{d-1}\mleft(d+\sqrt{d-1}\mright)^2, \quad \gamma = d(d-1)(d + 2\sqrt{d-1}).
    \]
    The left hand side of \cref{lem:polynomial_bound} can be rewritten as
    \[
        \frac{1}{n}\sum_i \lambda_i^4 - \alpha \cdot \frac{1}{n}\sum_i \lambda_i^2 + \beta \avge - \gamma,
    \]
    which by \cref{lem:moments} is at least
    \begin{equation} \label{eq:polynomial_bound_inequality}
        2\dbar^2 - \dbar - \alpha \dbar + \beta \avge - \gamma,
    \end{equation}
    where $\dbar$ is the average degree of $G$.
    
    Notice that the quadratic function $x \mapsto 2x^2 - (1+\alpha)x$ is decreasing on the interval $(-\infty, (1+\alpha)/4)$. One checks that $(1+\alpha)/4 > d$ for $d \ge 3$. Since $\dbar \le d < (1+\alpha)/4$, \eqref{eq:polynomial_bound_inequality} is at least $2d^2 - d - \alpha d + \beta \avge - \gamma$, which is equal to the right-hand side of \eqref{eq:polynomial_bound}.
\end{proof}

\begin{proof}[Proof of \cref{thm:avg_energy}]
    Let $f$ be the polynomial defined in \cref{lem:polynomial_bound}. Since $\abs{\lambda_i} \le d$ for every $i$, the definition of $f$ gives $f(\abs{\lambda_i}) \le 0$. Substituting this into \eqref{eq:polynomial_bound} completes the proof.
\end{proof}

The bound on the median eigenvalues in \cref{cor:median_energy_bound} follows from the following result.

\begin{theorem}[Theorem 3.1 of Li, Li, Shi, and Gutman \cite{LLSG13}] \label{thm:median_energy_bound}
    For every graph, the median eigenvalues are at most the average energy of the graph in absolute value.
\end{theorem}

\begin{proof}[Proof of \cref{cor:median_energy_bound}]
    It follows immediately from \cref{thm:avg_energy,thm:median_energy_bound}.
\end{proof}

Although the original statement of \cref{thm:median_energy_bound} is for connected graphs, its proof does not require this assumption. For completeness, we provide a streamlined proof.

\begin{proof}[Proof of \cref{thm:median_energy_bound}]
    Let $G$ be a graph on $n$ vertices, let $\lambda_1 \ge \dots \ge \lambda_n$ be its eigenvalues, and let $\lah$ and $\lal$ be the median eigenvalues of $G$. Since $\sum_i \lambda_i = 0$ by \cref{lem:moments}, we have
    \[
    \avge = \frac{2}{n} \sum_{i \colon \lambda_i \ge 0} \lambda_i = \frac{-2}{n}\sum_{i \colon \lambda_i \le 0} \lambda_i.
    \]
    Recall that $\lah$ and $\lal$ are the median eigenvalues of $G$, where $h = \lfloor (n+1)/2\rfloor$ and $\ell = \lceil (n+1)/2\rceil$. We first prove that $\lah \le \avge$. If $\lah \ge 0$, then $h \ge n/2$, and hence
    \[
        \avge = \frac{2}{n} \sum_{i \colon \lambda_i \ge 0} \lambda_i \ge \frac{2h \lah}{n} \ge \lah.
    \]
    Otherwise, if $\lah < 0$, then $\lah \le \avge$ holds trivially. The proof of $\lal \ge -\avge$ is analogous.
\end{proof}

\section{\texorpdfstring{Lower bound when $d$ is at least $75$}{Lower bound when d is at least 75}} \label{sec:lower_bound_d_large}

Finally, we present the proof of \cref{thm:lower_bound_d_large}, which reuses a couple of ideas from the proof of \cref{thm:lower_bound_d-1_perfect_square}. The proof of the following optimization problem is deferred to \cref{sec:optimization}.

\begin{proposition} \label{prop:product_bound}
    Fix an integer $d \ge 75$. Set \[
        \eps_0 = \sqrt{d-1} \quadand \eps_1 = \sqrt{d-1} + \frac{1}{d + \sqrt{d-1}}.
    \] Then there exists $\delta \in (\eps_0, d)$ such that the maximum value of 
    \[
        (2\eps_0)^{x + y + z}\mleft(\frac{\eps_0+\eps_1}{2\eps_0}\mright)(\eps - \eps_0)^x \mleft(\frac{\alpha(\delta) \cdot (\eps_1-\eps)}{y}\mright)^{y/2}(d - \eps_0)^z
    \]
    under the constraint that
    \[
        \eps_0 \le \eps \le \eps_1, \quad x \ge \frac{1}{2}, \quad 0 \le y \le \frac{1}{2} - z, \quad 0 \le z \le \frac{1}{(\delta-\eps_0)^2},
    \] is strictly less than $1$, where $\alpha(\delta)$ is defined by
    \[
        \alpha(\delta) = \frac{2\eps_0(d+\eps_0)^2}{(d-\delta)(d + 2\eps_0 + \delta)}.
    \]
\end{proposition}

\begin{proof}[Proof of \Cref{thm:lower_bound_d_large}]
    Fix an integer $d \ge 75$. Let $G$ be a graph of maximum degree $d$, let $\lambda_1 \ge \lambda_2 \ge \cdots \ge \lambda_n$ be the eigenvalues of its adjacency matrix $A_G$, let $\lah$ and $\lal$ be the median eigenvalues of $G$, and let $\eps$ be the average energy of $G$. We define the constants $\eps_0$ and $\eps_1$ as in \cref{prop:product_bound}.
    Assume for the sake of contradiction that the lower median eigenvalue $\lal$ is less than $-\eps_0$. By \cref{thm:median_energy_bound}, we are done as soon as $\eps \le \eps_0$. Hereafter, in view of \cref{thm:avg_energy}, we assume that
    \begin{equation}\label{eq:average_energy inequality}
        \eps_0 \le \eps \le \eps_1.
    \end{equation}
    
    Let $I$ denote the set of indices $i$ such that $\abs{\lambda_i} \neq \eps_0$. Since $\prod_{i \in I} (x - \lambda_i^2) \in \mathbb{Z}[x]$, we deduce that $\prod_{i \in I} \mleft(\eps_0^2 - \lambda_i^2\mright)$ is a nonzero integer.
    Consequently,
    \begin{equation} \label{eq:product_at_least_one}
        \prod_{i \in I} \abs*{\eps_0^2 - \lambda_i^2} \ge 1.
    \end{equation}
    We shall derive a contradiction by showing that the last inequality fails whenever $d \ge 75$. We factor the above product as
    \[
        A := \prod_{i \in I} \abs*{\eps_0 + \abs{\lambda_i}} \quadand
        B := \prod_{i \in I} \abs*{\eps_0 - \abs{\lambda_i}}.
    \]
    
    We first estimate $A$. Applying the inequality of arithmetic and geometric means, we obtain
    \begin{equation} \label{eq:A_estimate_1}
        A \le \mleft(\eps_0 + \frac{1}{\abs{I}}\sum_{i \in I} \abs{\lambda_i}\mright)^{\abs{I}}.
    \end{equation}
    Since $\lal < -\eps_0$, we have $\sset{\ell, \dots, n} \subseteq I$, and so $\abs{I} \ge n/2$. Since
    \[
        \eps = \frac{1}{n}\mleft(\sum_{i \in I} \abs{\lambda_i} + \sum_{i \notin I} \abs{\lambda_i}\mright) = \frac{1}{n}\mleft(\sum_{i \in I} \abs{\lambda_i} + (n - \abs{I})\eps_0\mright),
    \]
    we obtain
    \[
        \eps_0 + \frac{1}{\abs{I}} \sum_{i \in I} \abs{\lambda_i} = \eps_0 + \frac{1}{\abs{I}}\mleft(n\eps - \mleft(n - \abs{I}\mright)\eps_0\mright) = 2\eps_0 + \frac{n}{\abs{I}}(\eps - \eps_0).
    \]
    One checks that the function $x \mapsto (1 + 1/x)^x$ is increasing on $[0, \infty)$. Since $\abs{I} \le n$, we have the estimate
    \begin{multline} \label{eq:A_estimate}
        A \le \mleft(2\eps_0 + \frac{n}{\abs{I}}(\eps - \eps_0)\mright)^{\abs{I}} = (2\eps_0)^{\abs{I}} \mleft(1 + \frac{\eps - \eps_0}{2\eps_0}\cdot\frac{n}{\abs{I}}\mright)^{\abs{I}} \le (2\eps_0)^{\abs{I}} \mleft(1 + \frac{\eps - \eps_0}{2\eps_0}\mright)^n \\
        \stackrel{\eqref{eq:average_energy inequality}}{\le} (2\eps_0)^{\abs{I}} \mleft(1 + \frac{\eps_1 - \eps_0}{2\eps_0}\mright)^n = (2\eps_0)^{\abs{I}} \mleft(\frac{\eps_0+\eps_1}{2\eps_0}\mright)^n.
    \end{multline}
    
    We now turn to $B$. Let $J = \dset{i}{\lambda_i < -\eps_0}$. Notice that $J \subseteq \sset{\ell, \dots, n}$, and in particular,
    \begin{equation} \label{eqn:J_at_least_half}
        J \subseteq I \quadand \abs{J} \ge n/2.
    \end{equation}
    We estimate part of $B$ corresponding to $J$ using the inequality of arithmetic and geometric means:
    \[
        B_J := \prod_{i \in J} \abs*{\eps_0 - \abs{\lambda_i}} = \prod_{i \in J} \mleft(\abs{\lambda_i} - \eps_0 \mright) \le \mleft(\frac{1}{\abs{J}}\sum_{i \in J} \abs{\lambda_i} - \eps_0\mright)^{\abs{J}}.
    \]
    Since $\lambda_i$ with $i \in J$ have the same sign and $\sum_i \lambda_i = 0$ by \cref{lem:moments}, we have
    \[
        \sum_{i \in J} \abs{\lambda_i} = \abs*{\sum_{i \in J} \lambda_i} = \frac{1}{2}\mleft(\abs*{\sum_{i \in J} \lambda_i} + \abs*{\sum_{i \notin J} \lambda_i}\mright) \le \frac{1}{2}\sum_i \abs{\lambda_i} = \frac{1}{2} n\eps \le \abs{J}\eps,
    \]
    which implies that
    \begin{equation}\label{eq:B_J_estimate}
        B_J \le \mleft(\eps - \eps_0\mright)^{\abs{J}}.
    \end{equation}

    We shall further divide $I \setminus J$ into two subsets $K$ and $L$ as follows:
    \[
        K = \sset{i \in I \setminus J : \abs{\lambda_i} \le \delta} \quadand L = \sset{i \in I \setminus J : \abs{\lambda_i} > \delta},
    \]
    where $\delta \in (\eps_0, d)$ is a parameter to be determined later, and it depends only on $d$.
    
    To estimate the part of $B$ corresponding to $K$, we apply the inequality of arithmetic and geometric means:
    \begin{equation} \label{eq:B_K_estimate_1}
        B_K := \prod_{i \in K} \abs*{\eps_0 - \abs{\lambda_i}} = \mleft(\prod_{i \in K} \mleft(\eps_0 - \abs{\lambda_i}\mright)^2\mright)^{1/2} \le \mleft(\frac{1}{\abs{K}}\mleft(\sum_{i \in K} \mleft(\eps_0 - \abs{\lambda_i}\mright)^2\mright)\mright)^{\abs{K}/2}.
    \end{equation}
    Define $g(x)$ as follows:
    \[
        g(x) = (d - x)\mleft(\eps_0 - x\mright)^2\mleft(x + d + 2\eps_0\mright).
    \]
    Clearly $g(x) = -f(x)$, where $f(x)$ is defined in \cref{lem:polynomial_bound}, and $g(x) \ge 0$ for all $x \in [0, d]$. By \cref{lem:polynomial_bound}, we have
    \begin{equation}\label{eq:g_estimate}
        \frac{1}{n}\sum_i g(\abs{\lambda_i}) \le 2\eps_0\mleft(d+\eps_0\mright)^2 \mleft(\eps_1 - \eps\mright).
    \end{equation}
    Observe that the quadratic function $x \mapsto (d-x)(x+d+2\eps_0)$ is decreasing on $(0,d)$. For each $i\in K$, since $0 \le \abs{\lambda_i} \le \delta < d$, we have
    \[
        \mleft(\eps_0 - \abs{\lambda_i}\mright)^2 = \frac{g(\abs{\lambda_i})}{(d - \abs{\lambda_i})(\abs{\lambda_i} + d + 2\eps_0)} \le\frac{g(\abs{\lambda_i})}{(d-\delta)(\delta + d + 2\eps_0)}.
    \]
    Summing over $i\in K$, we obtain
    \begin{multline*}
        \sum_{i\in K} \mleft(\eps_0 - \abs{\lambda_i}\mright)^2 \le \frac{\sum_{i \in K} g(\abs{\lambda_i})}{(d-\delta)(\delta + d + 2\eps_0)} \le \frac{\sum_i g(\abs{\lambda_i})}{(d-\delta)(\delta + d + 2\eps_0)} \\
        \stackrel{\eqref{eq:g_estimate}}{\le} \frac{2\eps_0\mleft(d+\eps_0\mright)^2}{(d-\delta)(\delta + d + 2\eps_0)} \mleft(\eps_1 - \eps\mright) n = \alpha(\delta) \cdot (\eps_1 - \eps) \cdot n,
    \end{multline*}
    where $\alpha(\delta)$ is defined in \cref{prop:product_bound}.
    Substituting the last estimate into \eqref{eq:B_K_estimate_1}, we obtain that
    \begin{equation}\label{eq:B_K_estimate}
        B_K \le \mleft(\frac{\alpha(\delta) \cdot (\eps_1 - \eps) n}{\abs{K}}\mright)^{\abs{K}/2}.
    \end{equation}
    
    To estimate part of $B$ corresponding to $L$, we simply use $\abs{\lambda_i} \le d$ to get that
    \begin{equation}\label{eq:B_L_estimate}
        B_L \le (d - \eps_0)^{\abs{L}}.
    \end{equation}
    We also estimate $\abs{L}$ here. Using the fact that $\sum_i \lambda_i^2$ equals the degree sum of $G$, which is bounded by $nd$, we obtain that
    \[
        \sum_{i \in L}(\eps_0 - \abs{\lambda_i})^2 \le \sum_i(\eps_0 - \abs{\lambda_i})^2 \le n\eps_0^2 - 2\eps_0\sum_i \abs{\lambda_i} + \sum_i \lambda_i^2 \le n(2d-1-2\eps_0\eps) \stackrel{\eqref{eq:average_energy inequality}}{\le} n(2d-1-2\eps_0^2) = n.
    \]
    Since $\abs{\lambda_i} > \delta > \eps_0$ for $i \in L$, the above inequality implies that
    \begin{equation}\label{eq:B_L_estimate_2}
        \abs{L} \le \frac{n}{(\delta-\eps_0)^2}.
    \end{equation}
    
    Set $x = \abs{J}/n$, $y = \abs{K}/n$, and $z = \abs{L}/n$. Then $\abs{I}/n = x + y + z \le 1$. In view of \eqref{eq:average_energy inequality}, \eqref{eqn:J_at_least_half}, and \eqref{eq:B_L_estimate_2}, the parameters $\eps, x, y, z$ satisfy the constraints in \cref{prop:product_bound}. Combining the bounds \eqref{eq:A_estimate}, \eqref{eq:B_J_estimate}, \eqref{eq:B_K_estimate}, and \eqref{eq:B_L_estimate}, we obtain that
    \begin{align*}
        \mleft(\prod_{i \in I} \abs*{d-1 - \lambda_i^2}\mright)^{1/n} & = \mleft(A B_J B_K B_L\mright)^{1/n} \\
        & \le (2\eps_0)^{x+y+z} \mleft(\frac{\eps_0 + \eps_1}{2\eps_0}\mright) \mleft(\eps - \eps_0\mright)^x \mleft(\frac{\alpha(\delta)\cdot (\eps_1-\eps)}{y}\mright)^{y/2} \mleft(d - \eps_0\mright)^z.
    \end{align*}
    According to \cref{prop:product_bound}, there exists $\delta \in (\eps_0, d)$ such that the right hand side of the last inequality is strictly less than $1$, which contradicts \eqref{eq:product_at_least_one}.
\end{proof}

\section{Concluding remarks} \label{sec:further_remarks}

Our results confirm Mohar's conjecture for all but $64$ values of $d$ --- the exceptions are $d \in \sset{4, \dots, 74} \setminus \sset{5, 10, 17, 26, 37, 50, 65}$. In particular, it would be interesting to settle his conjecture for $d = 4$.

\begin{conjecture}
    The median eigenvalues of any graph of maximum degree $4$ are at least $-\sqrt{3}$.
\end{conjecture}

Inspired by the result of \cite{AJJ25} on subcubic graphs, it is natural to conjecture that the bound can be strengthened by excluding the incidence graph of the projective plane of order $3$.

\begin{conjecture}
    The median eigenvalues of any graph of maximum degree $4$, except for a vertex-disjoint union of incidence graphs of projective planes of order $3$, are at most $\sqrt{2}$ in absolute value.
\end{conjecture}

For graphs of maximum degree $5$, \cref{thm:upper_bound,thm:lower_bound_d-1_perfect_square} already imply that the median eigenvalues are bounded in absolute value by $2$. We note that the incidence graph of the projective plane of order $4$ is not the only connected $5$-regular graph attaining this bound. Indeed, the Cayley graph of the group $\mathbb{Z}/12\mathbb{Z}$ generated by $\sset{\pm 3, \pm 4, 6}$ is a connected $5$-regular graph whose median eigenvalues are $-2$ and $+1$.

A significant related result is due to Mohar and Tayfeh-Rezaie~\cite{MT15}, who proved that for every integer $d \ge 3$, the median eigenvalues of any \emph{bipartite} graph $G$ of maximum degree $d$ are at most $\sqrt{d-2}$ in absolute value, unless $G$ is the vertex-disjoint union of incidence graphs of projective planes of order $d-1$, in which case the median eigenvalues are $\pm\sqrt{d-1}$.

It is natural to relax the condition on maximum degree to one on average degree. From \cref{thm:mc-clelland-bound,thm:median_energy_bound}, we obtain the following bound on the median eigenvalues.

\begin{corollary}
    The median eigenvalues of any graph of average degree $d$ are at most $\sqrt{d}$ in absolute value. \qed
\end{corollary}

\cref{thm:lower_bound_triangle_free} suggests a stronger conjecture that improves the bound from $\sqrt{d}$ to $\sqrt{d-1}$.

\begin{conjecture}
    For every real number $d \ge 2$, the median eigenvalues of any graph of average degree at most $d$ are at most $\sqrt{d-1}$ in absolute value.
\end{conjecture}

We point out that the conjecture fails for $d \in [1,2)$. For example, the vertex-disjoint union of $a$ triangles and $b$ single edges has average degree $(6a+2b)/(3a+2b)$, but the median eigenvalues are $1$ in absolute value.

\section*{Acknowledgements}

The authors gratefully acknowledge the Simons Laufer Mathematical Sciences Institute (SLMath) for supporting travel and accommodation during the program \emph{Algebraic and Analytic Methods in Combinatorics}.

\bibliographystyle{plain}
\bibliography{homolumo}

\appendix

\section{Optimization} \label{sec:optimization}

\begin{proof}[Proof of \cref{prop:product_bound}]
    Fix an integer $d \ge 75$. We shall choose $\delta \in (\eps_0, d)$ later so that
    \begin{equation}\label{eq:delta_constraint}
        \alpha(\delta) \cdot (\eps_1-\eps_0) \ge 1/5 \quadand \delta - \eps_0 \ge \sqrt2.
    \end{equation}
    
    First, when $\eps, y, z$ are fixed, the function
    \begin{equation}\label{eq:first_function}
        (2\eps_0)^{x+y+z} \mleft(\frac{\eps_0 + \eps_1}{2\eps_0}\mright) \mleft(\eps - \eps_0\mright)^x \mleft(\frac{\alpha(\delta) \cdot (\eps_1-\eps)}{y}\mright)^{y/2}(d - \eps_0)^z
    \end{equation}
    is proportional to $(2\eps_0)^x(\eps-\eps_0)^x$, which is a decreasing function of $x$ because
    \[
        2\eps_0(\eps - \eps_0) \le 2\eps_0(\eps_1 - \eps_0) = \frac{2\sqrt{d-1}}{d + \sqrt{d-1}} < 1.
    \]
    Therefore, the maximum value of \eqref{eq:first_function} is attained when $x = 1/2$. Under this assumption, the function becomes
    \[
        (2\eps_0)^{1/2+y+z} \mleft(\frac{\eps_0 + \eps_1}{2\eps_0}\mright) \mleft(\eps - \eps_0\mright)^{1/2} \mleft(\frac{\alpha(\delta) \cdot (\eps_1-\eps)}{y}\mright)^{y/2}(d - \eps_0)^z.
    \]
    
    Second, when $y$ and $z$ are fixed, the above function is proportional to $(\eps-\eps_0)^{1/2}(\eps_1-\eps)^{y/2}$, which attains its maximum when $\eps = (y\eps_0+\eps_1)/(y+1)$. Under this assumption, the function becomes
    \begin{equation}\label{eq:second_function}
        (2\eps_0)^{1/2+y+z} \mleft(\frac{\eps_0 + \eps_1}{2\eps_0}\mright) \mleft(\frac{\eps_1-\eps_0}{y+1}\mright)^{1/2} \mleft(\frac{\alpha(\delta) \cdot (\eps_1-\eps_0)}{y+1}\mright)^{y/2}(d - \eps_0)^z.
    \end{equation}
    
    Third, when $z$ is fixed, the above function is proportional to \[
        (2\eps_0)^y\mleft(\frac{\alpha(\delta)\cdot (\eps_1-\eps_0)}{y+1}\mright)^{(y+1)/2}
    \]
    One checks that if $a^2 b > 3e/2$, then the function $y \mapsto a^{y}(b/(y+1))^{(y+1)/2}$ is increasing on $[0, 1/2]$. Since $(2\eps_0)^2 \cdot \alpha(\delta)\cdot (\eps_1-\eps_0) \ge 4(d-1)/5 \ge 57$ by \eqref{eq:delta_constraint}, the maximum value of \eqref{eq:second_function} is attained when $y = 1/2-z$. Under this assumption, the function becomes
    \begin{equation}\label{eq:third_function}
        (\eps_0 + \eps_1) \mleft(\frac{\eps_1-\eps_0}{3/2-z}\mright)^{1/2} \mleft(\frac{\alpha(\delta) \cdot (\eps_1-\eps_0)}{3/2-z}\mright)^{1/4-z/2}(d - \eps_0)^z.
    \end{equation}

    Fourth, the above function is proportional to
    \[
        \frac{(d - \eps_0)^z}{(3/2-z)^{3/4-z/2}}.
    \]
    One checks that if $a > e^{-1/2}$, then the function $z \mapsto a^z/(3/2-z)^{3/4-z/2}$ is increasing on $[0, 1/2]$. Since $d - \eps_0 \ge 60$ for $d \ge 75$ and $1/(\delta - \eps_0)^2 \le 1/2$ by \eqref{eq:delta_constraint}, the maximum value of \eqref{eq:third_function} is attained when $z = 1/(\delta - \eps_0)^2 =: z^*$. Under this assumption, the function becomes
    \begin{equation}\label{eq:fourth_function}
        (\eps_0 + \eps_1) \mleft(\frac{\eps_1-\eps_0}{3/2-z^*}\mright)^{1/2} \mleft(\frac{\alpha(\delta) \cdot (\eps_1-\eps_0)}{3/2-z^*}\mright)^{1/4-z^*/2}(d - \eps_0)^{z^*}.
    \end{equation}
    
    Suppose that $75 \le d \le 139$. We choose $\delta \in (\eps_0, d)$ so that $\alpha(\delta) \cdot (\eps_1 - \eps_0) = 1/4$, which is a quadratic equation in $\delta$. Solving this equation gives $\delta = \sqrt{(d + \eps_0)(d-7\eps_0)} - \eps_0$. One checks numerically for every $d \in \sset{75, \dots, 139}$ that \eqref{eq:delta_constraint} is satisfied and \eqref{eq:fourth_function} is strictly less than $1$.

    Suppose that $d \ge 140$. We choose $\delta \in (\eps_0, d)$ so that $\alpha(\delta) \cdot (\eps_1 - \eps_0) = 1/5$, which is a quadratic equation in $\delta$. Solving this equation gives $\delta = \sqrt{(d + \eps_0)(d-9\eps_0)} - \eps_0$. One checks that
    \begin{equation}\label{eq:delta_inequality}
        \delta - \eps_0 \ge d / 3 \text{ for } d \ge 140,
    \end{equation}
    and so \eqref{eq:delta_constraint} is satisfied. Since $\alpha(\delta) \cdot (\eps_1 - \eps_0) = 1/5$, the maximum value of \eqref{eq:fourth_function} simplifies to
    \begin{multline*}
        (\eps_0 + \eps_1) \mleft(\frac{\eps_1-\eps_0}{3/2-z^*}\mright)^{1/2} \mleft(\frac{1}{15/2-5z^*}\mright)^{1/4-z^*/2}(d - \eps_0)^{z^*} \\
        = \frac{\eps_0 + \eps_1}{\sqrt{d+\eps_0}} \cdot \frac{\sqrt{5}}{(15/2-5z^*)^{3/4-z^*/2}} \cdot \exp\mleft(\frac{\ln(d-\eps_0)}{(\delta-\eps_0)^2}\mright).
    \end{multline*}
    We estimate the three factors in the last expression separately. First, we have
    \[
        \frac{\eps_0 + \eps_1}{\sqrt{d+\eps_0}} \le \frac{2\sqrt{d}+1/d}{\sqrt{d}} = 2 + 1/\mleft(d\sqrt{d}\mright) \le 2 + 1/444.
    \]
    Second, since $z^* = 1/(\delta-\eps_0)^2 \le 9/d^2 \le 1/2000$ via \eqref{eq:delta_inequality}, we have
    \[
        \frac{\sqrt{5}}{(15/2-5z^*)^{3/4-z^*/2}} \le \frac{\sqrt{5}}{(15/2-1/400)^{3/4-1/4000}} \le \frac{1}{2+1/40}.
    \]
    Third, since $d \mapsto \ln d / d^2$ is a decreasing function on $(\sqrt{e}, \infty)$, we have that
    \[
        \exp\mleft(\frac{\ln(d-\eps_0)}{(\delta-\eps_0)^2}\mright) \stackrel{\eqref{eq:delta_inequality}}{\le} \exp\mleft(\frac{9\ln d}{d^2}\mright) \le \exp\mleft(\frac{9\ln 140}{140^2}\mright) \le \frac{111}{110}.
    \]
    Therefore, the product of the three factors is at most
    \[
        \frac{2+1/444}{2+1/40} \cdot \frac{111}{110} = \frac{889}{891}
    \]
    which is strictly less than $1$.
\end{proof}

\end{document}